\documentclass[11pt]{article}

\usepackage{amssymb}
\usepackage{amsfonts}
\usepackage{eucal}
\usepackage{amsmath}
\newtheorem{theorem}{Theorem}[section]

\newtheorem{definition}[theorem]{Definition}

\newtheorem{lemma}[theorem]{Lemma}
\newtheorem{remark}[theorem]{Remark}

\newenvironment{proof}[1][Proof]{\textbf{#1.} }{\ \rule{0.5em}{0.5em}}
\input pictex.sty

 \usepackage{cancel}
 \usepackage{xcolor}
 
 \usepackage{amsmath,color,ulem}
\newcommand{\stkout}[1]{\ifmmode\text{\sout{\ensuremath{#1}}}\else\sout{#1}\fi}

\DeclareMathOperator\Var{Var}
\DeclareMathOperator\Cov{Cov}

\begin{document}

 \title{New asymptotic results for generalized Oppenheim expansions  }
 	\author{Rita Giuliano \footnote{Dipartimento di
		Matematica, Universit\`a di Pisa, Largo Bruno
		Pontecorvo 5, I-56127 Pisa, Italy (email: rita.giuliano@unipi.it)}~~and Milto Hadjikyriakou\footnote{School of Sciences, University of Central Lancashire, Cyprus campus, 12-14 University Avenue, Pyla, 7080 Larnaka, Cyprus (email:
		mhadjikyriakou@uclan.ac.uk). Part of this work was conducted while the author was a visiting scholar at the University of Cyprus.}} 	
 \maketitle

\begin{abstract}
In this work, we study convergence in probability and almost sure convergence for weighted partial sums of random variables that are related to the class of generalized Oppenheim expansions. It is worth noting that the random variables under study have infinite mean and the results are obtained without any dependence assumptions.
\end{abstract}


\medskip


\section{Introduction}
In the case of independent and identically distributed random variables $(X_n)_{n\geq 1}$ with finite nonzero mean $\mu$, a celebrated asymptotic result by Kolmogorov ensures the almost sure convergence to 1 of the random quantity
\[
\frac{1}{n\mu}\sum_{j=1}^{n}X_j.
\]
Weak and strong asymptotic results for random variables with either zero or infinite mean can also be obtained,  however in these cases the procedure requires the correct adjustment for the involved weights. Asymptotic results of this form have been studied extensively by various authors, see for example \cite{A2000,A2007,A2008,A2012,AM2018,CR1961} and references therein, and more recently by \cite{GH2020} and \cite{BB2022}.

The framework of this work has been introduced and generalized in \cite{G2018} and \cite{GH2020} respectively and is described as follows: let $(B_n)_{n\geq  1}$ be a sequence of integer valued random  variables defined on $(\Omega, \mathcal{A}, P)$, where $\Omega =[0,1]$, $\mathcal{A}$ is the $\sigma$-algebra of the Borel subsets of $[0,1]$ and $P$ is the Lebesgue measure on $[0,1]$. Let $\{F_n, n\geq 1\}$ be a sequence of probability distribution functions defined on $[0,1]$ with  $F_n(0)=0$, $\forall n$ and moreover let $\varphi_n:\mathbb{N}^*\to \mathbb{R}^+$ be a sequence of functions. Furthermore, let $(q_n)_{n\geq  1}$ with $q_n=q_n(h_1, \dots, h_n)$ be a sequence of nonnegative numbers (i.e. possibly depending on the $n$ integers $h_1, \dots, h_n$) such that, for $h_1 \geq  1$ and $h_j\geq  \varphi_{j-1}(h_{j-1})$, $j=2, \dots, n$ we have
\[
P\big(B_{n+1}=h_{n+1}|B_{n}=h_{n}, \dots, B_{1}=h_{1}\big)= F_n(\beta_n)-F_n(\alpha_n),
\]
where 
\[
\alpha_n=\delta_n(h_n, h_{n+1}+1, q_n)  ,\quad \beta_n=\delta_n(h_n, h_{n+1}, q_n)\quad\mbox{with}\quad\delta_j(h,k, q) = \frac{ \varphi_j (h )(1+y )}{k+\varphi_j (h ) q  }.
\]
Let $Q_n= q_n(B_1, \dots, B_n)$ and define
\begin{equation}
\label{Rdef}R_{n}= \frac{ B_{n+1}+\varphi_n(B_n) Q_n}{\varphi_n(B_n)(1+Q_n) }= \frac{1}{\delta_n(B_n, B_{n+1}, Q_n)}.
\end{equation}
Particular cases of these random variables have been studied extensively over the years; for example, the L\"uroth series was studied in \cite{G1976} and \cite{L1883}, the Engel series was studied in \cite{E1913} and \cite{S1974} while the cases of Engel continued fraction expansions and the Sylvester series were explored in \cite{HKS2002} and \cite{P1960} respectively. More recently, in \cite{G2018}, a weak law has been established for the sequence $\displaystyle\frac{1}{n\log n}\sum_{k=1}^{n}R_k,$ while in \cite{GH2020} exact weak and strong laws (i.e. convergence either in probability or almost surely to a positive constant) have been proven for the sequence $\displaystyle \frac{1}{b_n}\sum_{k=1}^{n}a_kR_k$ for suitably chosen sequences of positive numbers $(a_n)_{n\geq 1}$ and $(b_n)_{n\geq 1}$. Moreover, in \cite{FG2021}, conditions are identified  under which $R_n^{-1}$ converges in distribution and the sequence $\displaystyle\sum_{k=1}^{n}\log R_k$ satisfies a central limit theorem. Let $\bar{G}_n(x) = P(R_n>x)$. It can be proven that for any $y>0,\,  \bar{G}_n(x+y)\sim \bar{G}_n(x)$ as $x\to\infty$ which means that for every integer $n$ the distribution of random variable $R_n$  is a \textit{long-tailed} distribution. Long-tailed distributions are used in many different research areas such as in finance and actuarial, machine learning and artificial intelligence (see for example \cite{FKZ2011} for more on long-tailed distributions). Thus, the asymptotic behavior of the particular model is essential to be studied due to the wide applicability of long-tailed distributions. 

Throughout the paper, the notation $a_n\sim b_n$, will be used to denote $\displaystyle\lim_{n\to \infty}\frac{a_n}{b_n}=1$ while the constant $C$ will stand for a real number not necessarily the same in every appearance. Moreover, the symbol $I(A)$ denotes the indicator function of the set $A$, while $a\wedge b$ will be used to express the $\min\{a,b\}$.  The purpose of this work is to study the asymptotic behavior of 
\[
\frac{1}{b_n^p}\sum_{j=1}^{n}a_j(R_j-Eg(R_j))
\]
for suitably chosen contant $p$, sequences of positive numbers $(a_n)_{n\geq 1}$ and $(b_n)_{n\geq 1}$ and function $g(\cdot)$. This kind of results are in the same line as in \cite{N2020} and are of interest in the case of variables with infinite mean.

\medskip

The paper is structured as follows: Section 2 includes some preliminary results that are essential for the rest of the paper while Sections 3 provides a series of results related to convergence in probability under various assumptions. Finally, in Section 4
we discuss a result concerning almost sure convergence.

\section{Preliminaries}
The below result is proved in \cite{GH2020} (see Lemma 3 there).
\begin{lemma}
	Let $(R_n)_{n\geq 1}$ be as  in \eqref{Rdef}. Then, for any integer $n$ and for $x\geq 1$,
	\[
	E\left[F_n\left(\frac{\varphi_n(B_n)(1+Y_n)}{x\varphi_n(B_n)(1+Y_n)+1  }\right)\right] \leq P(R_n>x)\leq F_n\left(\frac{1}{x}\right).
	\]
\end{lemma}
The right inequality above can be interpreted as follows: let $U_n$ be a  random variable with distribution $F_n$ and define $Y_n = \displaystyle \frac{1}{U_n}$ for every integer $n$. Then
\begin{equation}
\label{stdom}P(R_n>x)\leq P(Y_n>x),\quad x\geq 1.
\end{equation}
i.e., the sequence $(R_n)_{n\geq 1}$ is \textit{stochastically dominated} by the sequence $(Y_n)_{n\geq 1}$. 

\medskip

The following result provides useful moment inequalities that link the variables $R_n$ and $Y_n$. These inequalities are modified versions of known results for the needs of our framework (see for example, Lemma 1 of \cite{AR1987} and Lemma 3 of \cite{ART1989}).

\begin{lemma}
	\label{trunmom}Let $(R_n)_{n\geq 1}$ and $(Y_n)_{n\geq 1}$ be as in \eqref{Rdef} and \eqref{stdom} respectively. Then, for every $q>0$, $t\geq 1$ and every integer $n$,
	\begin{equation}
	\label{trun1}E(R_n^qI(R_n\leq t)) \leq 1+ t^qP(Y_n>t)+E(Y_n^qI(Y_n\leq t))
	\end{equation}
	\begin{equation}
	\label{trun2}E(R_n^qI(R_n>t)) \leq E(Y_n^qI(Y_n> t)).
	\end{equation}
\end{lemma}
\begin{proof}
Let $X$ be a random variable with distribution function $F$ ($\overline F = 1-F$). Let $0\leq c \leq t$ and $q>0$. Then, integrating by parts, 
\begin{align*}
&q\int_c^t s^{q-1}P(X>s){\rm d} s= \big[s^q P(X>s)\big]_c^t - \int_c^ts^q {\rm d} \overline F (s)\\
	&=t^q P(X>t)-c^q P(X>c) +\int_c^ts^q {\rm d}   F(s)= t^q P(X>t)-c^q P(X>c) +\int s^qI(c\leq s \leq t){\rm d}   F(s)\\
	&=t^q P(X>t)-c^q P(X>c) +\int X^qI(c\leq X \leq t) {\rm d}P=   t^q P(X>t)-c^q P(X>c) +E\big[X^qI(c\leq X \leq t)\big].
\end{align*}
By taking $c=0$ we find
\begin{equation}\label{eq1}
q\int_0^t s^{q-1}P(X>s){\rm d} s= t^q P(X>t) +E\big[X^qI(0\leq X \leq t)\big]= t^q P(X>t) + E\big[X^qI(X \leq t)\big].
\end{equation}
Moreover, if $X$ is assumed to be a nonnegative random variable, the above expression gives 
\begin{align}
&\nonumber q\int_c^t s^{q-1}P(X>s){\rm d} s\leq t^q P(X>t)+E\big[X^qI(c\leq X \leq t)\big]\leq t^q P(X>t)+E\big[X^qI(0\leq X \leq t)\big]\\
&=\label{eq2} t^q P(X>t)+ E\big[X^qI(X \leq t)\big].
\end{align} 

\noindent Recall that $R_n \geq 1$. The expressions obtained above lead to the following calculations for $t \geq 1$, 
\begin{align*}&
E\big[R_n^q I(R_n\leq t)\big]\leq E\big[R_n^q I(R_n\leq t)\big]+ t^q P(R_n>t)= q\int_0^t s^{q-1}P(R_n>s){\rm d} s \\& =  q\int_0^1 s^{q-1}P(R_n>s){\rm d} s+ q\int_1^t s^{q-1}P(R_n>s){\rm d} s\leq 
q\int_0^1 s^{q-1} {\rm d} s+ q\int_1^t s^{q-1}P(Y_n>s){\rm d} s\\
& \leq\big[s^q\big]_0^1+ t^q P(Y_n>t)+ E\big[Y_n^qI(Y_n\leq t)\big]=1+ t^q P(Y_n>t)+ E\big[Y_n^qI(Y_n\leq t)\big],\end{align*}
where the first inequality follows from  \eqref{stdom} while the second one  derives from \eqref{eq2} with $Y_n$ in place of $X$.

\medskip

\noindent For \eqref{trun2}, first recall that for any nonnegative random variable $X$ we have that 
\[
EX = xP(X>x)+\int_{x}^{\infty}P(X>u){\rm d}u.
\]
Thus, for $t\geq1$, with $X= R_n^q I(R_n >t)$ and the change of variable $u=x^q$
\begin{eqnarray*}
	E(R_n^qI(R_n>t)) &=&t^qP(R_n>t)+q\int_{t}^{\infty}x^{q-1}P(R_n>x){\rm d}x\\
	&\leq& t^qP(Y_n>t)+q\int_{t}^{\infty}x^{q-1}P(Y_n>x){\rm d}x\\
	&=&E(Y_n^qI(Y_n>t))
\end{eqnarray*}
where the inequality is due to the stochastic dominance of $R_n$ from $Y_n$ for any integer $n$.
\end{proof}

\begin{lemma}
	\label{lemma1}Let $(R_n)_{n\geq 1}$ be as in \eqref{Rdef} with the related sequence of distributions $(F_n)_{n\geq 1}$ satisfying the below condition for some $\alpha>0$
	\begin{equation}
	\label{cond2F}\limsup_{x\to 0}\sup_{n\geq 1} \frac{F_n(x)}{x^{\alpha} } =: L <\infty.
	\end{equation}
	Let $(a_n)_{n\geq 1}$ and $(b_n)_{n\geq 1}$ be sequences of positive numbers such that
	\begin{equation}
		 \label{condconst}\sum_{j=1}^{n}a_j^{\alpha} = o(b_n^{\alpha}) \quad \mbox{ as } n \to \infty.
	\end{equation}
	Then,
	\begin{equation}
	\label{sumProb}\sum_{j=1}^{n}P\left(R_j>\frac{b_n}{a_j}\right)\to 0 \mbox{ as }n\to \infty.
	\end{equation}	
\end{lemma}
\begin{proof}
Let $L^\prime > L$. There exists $\delta >0$ such that  $ \sup_k \frac{F_k(x)}{x^\alpha}< L^\prime$ for every $x \in (0,\delta).$  Let $n_0$ be such that, for $n> n_0$,
$$\sum_{j=1}^n \Big(\frac{a_j}{b_n}\Big)^\alpha < \delta^\alpha \wedge 1.$$ Then, for every $n>n_0$ and  every $j= 1, \dots, n$,

\begin{align*}&
\frac{a_j}{b_n}= \Big\{\Big(\frac{a_j}{b_n}\Big)^\alpha\Big\}^\frac{1}{\alpha}\leq \Big\{\sum_{j=1}^n \Big(\frac{a_j}{b_n}\Big)^\alpha \Big\}^\frac{1}{\alpha}< \delta \wedge 1.
\end{align*}
Thus, for every $n>n_0$,
\begin{align*}&
\sum_{j=1}^n P\Big(R_j > \frac{b_n}{a_j}\Big)\leq \sum_{j=1}^n \sup_k \frac{F_k(\frac{a_j}{b_n})}{(\frac{a_j}{b_n})^\alpha}\Big(\frac{a_j}{b_n}\Big)^\alpha <  L^\prime\sum_{j=1}^n \Big(\frac{a_j}{b_n}\Big)^\alpha,
\end{align*}
and we conclude by passing to the limit as $n \to \infty$.
\end{proof}

\medskip

\begin{remark}
	A similar result has been obtained in \cite{BB2022} (see Lemma 2.2 (ii) there) for random variables $(X_n)_{n\geq 1}$ satisfying conditions
	\begin{equation}
	\label{condBB1} P(|X|>x) \asymp x^{-\alpha} 
	\end{equation}
	and
	\begin{equation}
	\label{condBB2} \limsup_{x\to\infty} \sup_{n\geq 1}x^{\alpha}P(|X_n|>x)<\infty.
	\end{equation}
for some  $0<\alpha\leq 1$. Note that the uniformity condition in \eqref{cond2F} is used in \cite{GH2020}. It is worth pointing out that since $P(Y_n>x) = F_n(x^{-1})$, relation \eqref{cond2F} leads to
	\[
	\limsup_{x\to\infty}\sup_{n\geq 1}x^{\alpha}P(Y_n>x)<\infty
	\]
	which implies infinite mean for the involved random variables for $\alpha\in (0,1]$.
\end{remark}
 
\medskip
 
\noindent Next, we define a new sequence of random variables based on truncation on $R_n$. Let $(a_n)_{n\geq 1}$ and $(b_n)_{n\geq 1}$ be sequences of positive numbers.  For $1\leq j\leq n$, we define,
\begin{equation}
\label{Rnjdef}R_{nj} = R_jI\left(R_j\leq\frac{b_n}{a_j}\right)+\frac{b_n}{a_j}I\left(R_j>\frac{b_n}{a_j}\right).
\end{equation}

\begin{lemma}
\label{expIneq}	Let $\{R_n\}_{n\geq 1}$ be as in \eqref{Rdef} and assume that the conditions of Lemma \ref{lemma1} are satisfied. Consider $p\geq\alpha$. Then for every $L'>L$, there exists $n_0\in\mathbb{N}$ such that for $n>n_0$ and $1\leq j\leq n$
	\begin{equation}
	\label{exp2}ER_{nj}^{p} \leq 
	\begin{cases}
	C_{L^\prime}\left(\dfrac{b_n}{a_j}\right)^{p-\alpha}, & p>\alpha\\
	D_{L^\prime}\left(\frac{b_n}{a_j}\right)^\alpha,  & p=\alpha,
	\end{cases}
	\end{equation}
where $C_{L^\prime} = 1+\dfrac{3p-2\alpha}{p-\alpha}L'$ and $D_{L^\prime} = 1+2L'$.
\end{lemma}
\begin{proof}
	First we start with $p>\alpha$. Then, for every $1\leq j\leq n$,
	\[
	R_{nj}^{p} = \left(\dfrac{b_n}{a_j}\right)^{p}I\left(R_j>\dfrac{b_n}{a_j}\right) + R_j^pI\left(R_j\leq\dfrac{b_n}{a_j}\right)
	\]
	which gives
		\[
	ER_{nj}^{p} = \left(\dfrac{b_n}{a_j}\right)^{p}P\left(R_j>\dfrac{b_n}{a_j}\right) + ER_j^pI\left(R_j\leq\dfrac{b_n}{a_j}\right).
	\]
Lemma \ref{trunmom} will be employed for the second term and hence
\[
ER_{nj}^{p} \leq \left(\dfrac{b_n}{a_j}\right)^{p}P\left(R_j>\dfrac{b_n}{a_j}\right) + 1+\left(\dfrac{b_n}{a_j}\right)^{p}P\left(Y_j>\dfrac{b_n}{a_j}\right) +E\left(Y_j^pI\left(Y_j\leq \dfrac{b_n}{a_j}\right)\right).
\]

\noindent Following the lines of the proof of Lemma \ref{lemma1} we can find $\nu$ such that for every $n>\nu$ and for every $1\leq j\leq n$ we have that $\displaystyle\frac{a_j}{b_n}\leq 1$ and therefore  \eqref{stdom} can be used leading to
	\begin{equation}
	\label{momentineq}ER_{nj}^{p}  \leq 1+2\left(\dfrac{b_n}{a_j}\right)^{p}P\left(Y_j>\dfrac{b_n}{a_j}\right) +E\left(Y_j^pI\left(Y_j\leq \dfrac{b_n}{a_j}\right)\right).
	\end{equation}
	Fix $L'>L$. Then, by \eqref{cond2F}, there is an integer $n_0>\nu$ such that $\forall n\geq n_0$ and $1\leq j\leq n$
	\begin{equation}
\label{ineqY}	P\left(Y_j>\dfrac{b_n}{a_j}\right) \leq L'\left(\dfrac{b_n}{a_j}\right)^{-\alpha}
	\end{equation}
	(see the proof of Lemma \ref{lemma1}). Thus, for $n \geq n_0$, $p-\alpha>0$ and $1\leq j\leq n$, again by \eqref{cond2F},
	\begin{eqnarray*}
	ER_{nj}^{p}&\leq& 1 +2L^\prime\left(\frac{b_n}{a_j}\right)^p\left(\frac{b_n}{a_j}\right)^{-\alpha} +  \int_{1}^{(b_n/a_j)^p}P(Y_j^p>t){\rm d}t\\
	&\leq& 1 +2L^\prime\left(\frac{b_n}{a_j}\right)^{p-\alpha}+L^\prime \int_{1}^{(b_n/a_j)^p}t^{-\alpha/p}{\rm d}t\\
	&  \leq& 1 +2L^\prime\left(\frac{b_n}{a_j}\right)^{p-\alpha}+\frac {pL^\prime}{p-\alpha}\left(\frac{b_n}{a_j}\right)^{p-\alpha}\\
	&\leq&C_{L^\prime}\left(\frac{b_n}{a_j}\right)^{p-\alpha}
\end{eqnarray*}
for $C_{L^\prime}= 1 + \frac{3p-2\alpha}{p-\alpha}L ^\prime$, since by the arguments used in the proof of Lemma \ref{lemma1} we have that  $$\left(\frac{b_n}{a_n}\right)^{p-\alpha}>1.$$ For the case where $p=\alpha$ and for sufficiently large $n$, by utilizing \eqref{momentineq} and \eqref{ineqY} we can write
\begin{eqnarray*}
	ER_{nj}^{p}&\leq&1 + 2L'+\int_{1}^{(b_n/a_j)^\alpha}P(Y_j^\alpha>t){\rm d}t\\
	&\leq&1+2L'+ L'\int_{1}^{(b_n/a_j)^\alpha}\frac{1}{t}{\rm d}t\\
	&\leq&1 +2L' +L'\log\left(\frac{b_n}{a_j}\right)^\alpha\\
	&\leq&\begin{cases}
		(1+2L')\left(\frac{b_n}{a_j}\right)^\alpha+L'\left(\frac{b_n}{a_j}\right)^\alpha\\
		(1+2L')\alpha\left(\frac{b_n}{a_j}\right)+\alpha L'\left(\frac{b_n}{a_j}\right)
	\end{cases}\\
&\leq&D_{L^\prime}\left(\left(\frac{b_n}{a_j}\right)^\alpha\wedge \alpha\left(\frac{b_n}{a_j}\right)\right),
\end{eqnarray*}	
where $D_{L^\prime} = 1+3L'$. Note that in order to obtain the expressions above we used that $\log x\leq x,\, \mbox{for } x>0$ and that for sufficiently large $n$ $$\left(\frac{b_n}{a_n}\right)^{\alpha}>1 \quad \mbox{and} \quad \alpha\left(\frac{b_n}{a_n}\right)>1.$$ 
\end{proof}

	

\begin{remark}
	In \cite{BB2022} (see Lemma 2.2(i) there) a similar moment inequality has been obtained for random variables that satisfy conditions \eqref{condBB1} and \eqref{condBB2} and it is used for obtaining asymptotic results for random variables that satisfy either a Rosenthal-type maximal inequality or a Marcinkiewicz-Zygmund type inequality. In our framework, the desired moment inequality is obtained by using the stochastic dominance as an instrumental result together with condition \eqref{cond2F} which refers to the random variables $(Y_n)_{n\geq 1}$. Moreover, the asymptotic results that are obtained in the next sections are proven without any assumptions on the dependence structure of the random variables $(R_n)_{n\geq 1}$ or by imposing any further conditions that are commonly used as tools in obtaining asymptotic results.
\end{remark}

\section{Weak laws of large numbers}
In this section, we provide some weak laws of large numbers for the sequence $(R_n)_{n\geq 1}$. Particularly, we identify conditions under which the convergence in probability is established for the sequence
\[
\frac{1}{b_n^{p}}\sum_{j=1}^{n}a_j(R_j-Eg(R_j)),
\]
where $p\geq 1$ and $g(\cdot)$ is a suitably chosen function. The symbol $\displaystyle\mathop {\longrightarrow }^P$ indicates convergence in probability.

\medskip

 \begin{theorem}
	\label{weak1}Let $(R_n)_{n\geq 1 }$ be as defined in \eqref{Rdef}. Assume that 
	\begin{enumerate}
	\item[i.] there exists $M<\infty$ such that $\forall j =1,2,\ldots$
	\begin{equation}
\label{F1cond}	F_j(x)- F_j(y)\leq M(x-y),\quad\mbox{for}\quad x>y;
	\end{equation}
	\item[ii.] there exists $\alpha>0$ and $L>0$ such that
\begin{equation}
\label{F2cond}	\lim_{x\to\infty}\sup_n\left | \dfrac{F_n(x)}{x^\alpha}-L\right |=0.
	\end{equation}
	\end{enumerate}
	Furthermore, assume that the sequences $(a_n)_{n\geq 1}$ and $(b_n)_{n\geq 1}$ satisfy condition \eqref{condconst} and that for some $p>1$, \, $n/b_n^{p-1} \to 0$ as $n \to \infty$.
\noindent Then,
	\[
	\frac{1}{b_n^p}\sum_{j=1}^{n}a_j(R_j-ER_{nj}) \mathop {\longrightarrow }^P 0\quad n\to \infty.
	\]
\end{theorem}
\begin{proof}
	First, we write  
	\[
	\frac{1}{b_n^p}\sum_{j=1}^{n}a_j(R_j-ER_{nj})=\frac{1}{b_n^p}\sum_{j=1}^{n}a_j(R_j-R_{nj})+\frac{1}{b_n^p}\sum_{j=1}^{n}a_j(R_{nj}-ER_{nj}).
	\]
	Thus, the result will follow by proving that, as $n\to \infty$,
	\begin{equation}
	\label{conv1}\frac{1}{b_n^p}\sum_{j=1}^{n}a_j(R_j-R_{nj}) \mathop {\longrightarrow }^P 0
	\end{equation}
	and
		\begin{equation}
	\label{conv2}\frac{1}{b_n^p}\sum_{j=1}^{n}a_j(R_{nj}-ER_{nj})\mathop {\longrightarrow }^P 0.
	\end{equation}
	For \eqref{conv1} we have that $\forall \epsilon>0$ 
	\[
	P\left(\frac{1}{b_n^p}\left | \sum_{j=1}^{n}a_j(R_j-R_{nj}) \right |>\epsilon\right)\leq P\left(\displaystyle\cup_{j=1}^{n}\{R_j \neq R_{nj}\}\right)\leq \sum_{j=1}^{n} P\left(R_j>\frac{b_n}{a_j}\right) \to 0,\quad n\to \infty,
	\]
because of Lemma \ref{lemma1}.
	
	\medskip
	
	\noindent For \eqref{conv2}, starting with Markov inequality,  we have that $\forall \epsilon>0$ 
	\[
	P\left(\frac{1}{b_n^p}\left | \sum_{j=1}^{n}a_j(R_{nj}- ER_{nj}) \right |>\epsilon\right)\leq\frac{1}{b_n^{2p}\epsilon^2}E\left(   \sum_{j=1}^{n}a_j(R_{nj}- ER_{nj})\right)^2.
	\]
	We denote $\xi_j =R_{nj}- ER_{nj} $. Then,
	\begin{eqnarray*}
	E\left( \sum_{j=1}^{n} a_j\xi_j\right)^2 &=& \sum_{j=1}^{n} a_j^2E\xi_j^2 +2\sum_{i=1}^{n}\sum_{j=1}^{i-1}a_ia_jE(\xi_i\xi_j)\\
	&\leq& \sum_{j=1}^{n} a_j^2E\xi_j^2 + C\sum_{i=1}^{n}\sum_{j=1}^{i-1}a_ia_j\left( \frac{b_n^2}{a_ia_j}\right)\\
	&=&\sum_{j=1}^{n} a_j^2E\xi_j^2 + Cb_n^2\sum_{i=1}^{n}\sum_{j=1}^{i-1}1\\
	&\leq&\sum_{j=1}^{n} a_j^2E\xi_j^2 + Cn^2b_n^2,
	\end{eqnarray*}
where the first inequality is due to Lemma 4 in \cite{GH2020}. Note that $E\xi_j^2 \leq ER_{nj}^2 \leq \left(\dfrac{b_n}{a_j}\right)^2$. Thus, 
\[
E\left( \sum_{j=1}^{n} a_j\xi_j\right)^2\leq nb_n^2 +Cn^2b_n^2 \leq Cn^2b_n^2
\]	
which leads to
\[
P\left(\frac{1}{b_n^p}\left | \sum_{j=1}^{n}a_j(R_j-ER_{nj}) \right |>\epsilon\right)\leq \dfrac{Cn^2}{b_n^{2(p-1)}}\to 0
\]
as $n\to \infty$, which concludes the proof.
\end{proof}

\begin{remark}
	It is worth pointing out that condition \eqref{F2cond} implies \eqref{cond2F}.
\end{remark}

\medskip 
\noindent The next result is an immediate consequence of the weak law proven above.

	 



\begin{theorem}
	\label{weak3}Let $(R_n)_{n\geq 1}$ be as in \eqref{Rdef}. Under the assumptions of Theorem \ref{weak1} for $\alpha =1$,	
		\[
		\lim_{n\to \infty}\frac{1}{b_n^p}\sum_{j=1}^{n}a_jR_j = 0 \quad\mbox{in probability}.
		\]
\end{theorem} 
\begin{proof}
	First we write 
	\begin{eqnarray*}
		\frac{1}{b_n^p}\sum_{j=1}^{n}a_jR_j &=& \frac{1}{b_n^p}\sum_{j=1}^{n}a_j(R_j-ER_{nj}) +	\frac{1}{b_n^p}\sum_{j=1}^{n}a_j\left(ER_{nj}- E\left(R_jI\left(R_j\leq \frac{b_n^p}{a_j}\right)\right)\right)\\
		&~~~~&+ \frac{1}{b_n^p}\sum_{j=1}^{n}a_jE\left(R_jI\left(R_j\leq \frac{b_n}{a_j}\right)\right)\\
		&=:& I_1+I_2+I_3.
	\end{eqnarray*}
	Observe that as $n\to \infty$, $I_1 \to 0$ because of Theorem \ref{weak1}. For $I_3$ notice that as $n\to \infty$,
	\[
	0\leq \frac{1}{b_n^p}\sum_{j=1}^{n}a_jE\left[R_jI\left(R_j\leq \frac{b_n}{a_j}\right)\right]\leq \frac{1}{b_n^p}\sum_{j=1}^{n}b_nP\left(R_j\leq \frac{b_n}{a_j}\right)\leq\dfrac{n}{b_n^{p-1}} \to 0
	\] 
and hence $I_3 \to 0$ as $n\to\infty$. The result follows by proving that $I_2$ also tends to $0$ as $n\to \infty$. First note that
	\[
	\sup_j\left(\frac{a_j}{b_n}\right)\leq \sum_{j=1}^{n}\frac{a_j}{b_n}\to 0\quad \mbox{as}\quad n\to \infty,
	\]
	which means that $\inf_j\frac{b_n}{a_j}\to\infty$. Moreover, 
	\[
	b_n = \dfrac{b_n}{\sup_ja_j}(\sup_ja_j )= \inf_j\frac{b_n}{a_j}(\sup_ja_j) \to \infty\quad \mbox{as}\quad n\to \infty.
	\]
	Thus, as $n\to \infty$
	\[ 
	\frac{1}{b_n^p}\sum_{j=1}^{n}a_j\left(ER_{nj}- E\left(R_jI\left(R_j\leq \frac{b_n}{a_j}\right)\right)\right) = \frac{1}{b_n^p}\sum_{j=1}^{n}a_j \left(\frac{b_n}{a_j}P\left(R_j >\frac{b_n}{a_j}\right)\right) =\frac{1}{b_n^{p-1}}\sum_{j=1}^{n} P\left(R_j >\frac{b_n}{a_j}\right)   \to 0 
	\]
	because of Lemma \ref{lemma1}.
\end{proof}

\begin{remark}
	Note that the convergence in probability for the sequence $(R_n)_{n\geq 1}$ has also been studied in \cite{GH2020} under different conditions (see for example Theorems 6 and 7 there).
\end{remark}

\medskip

\noindent In the next result we establish a moment inequality that will be used in the proof of the weak law that follows.

\begin{lemma}
Let $(R_n)_{n\geq 1}$ be as in \eqref{Rdef} and assume that \eqref{cond2F} is satisfied. Then, for sufficiently large $n$ and for some $\alpha>0$,
	\[
E\left(   \sum_{j=1}^{n}a_j(R_{nj}- ER_{nj})      \right)^2 \leq
\begin{cases}\displaystyle
Cnb_n^{2-\alpha}\sum_{j=1}^{n}a_j^\alpha, & 0<\alpha<2\\
\displaystyle Cnb_n\sum_{j=1}^{n}a_j, & \alpha =2\\
\displaystyle Cn\sum_{j=1}^{n}a_j^2, & \alpha>2,
\end{cases}
\]
where the symbol $C$ stands for a positive constant which may be different in every appearance.
\end{lemma}
\begin{proof}
	\begin{equation}
\label{boundsumE}	E\left(   \sum_{j=1}^{n}a_j(R_{nj}- ER_{nj})      \right)^2 \leq n   \sum_{j=1}^{n}a_j^2E(R_{nj}- ER_{nj})^2\leq n\sum_{j=1}^{n}a_j^2ER_{nj}^2.
	\end{equation}
	By using the change of variables $\sqrt{u} = t$ and since $R_j \geq 1$, for sufficiently large $n$ we have that
	\begin{align*}
	&ER_{nj}^2 = 2\int_{0}^{\infty} tP(R_{nj}>t){\rm d}t = 2 \int_{1}^{b_n/a_j}tP(R_j>t){\rm d}t\leq 2\int_{1}^{b_n/a_j}tF_j\left(\frac{1}{t}\right){\rm d}t\\
	&= 2\int_{a_j/b_n}^{1}\frac{F_j(t)}{t^3}{\rm d}t\leq C \int_{a_j/b_n}^{1}\frac{1}{t^{3-\alpha}}{\rm d}t,
	\end{align*}
	where the first inequality follows from \eqref{stdom} while for the second one we apply arguments similar to the ones used for the proof of Lemma \ref{lemma1}. Thus,
		\begin{align*}
	&ER_{nj}^2 =\begin{cases}
	\dfrac{C}{\alpha-2}\left(1 - \left(\frac{a_j}{b_n}\right)^{\alpha-2}\right), & \alpha\neq 2\\
	C\log\left(\frac{b_n}{a_j}\right), & \alpha =2
	\end{cases}
	\,= \,\begin{cases}
	\dfrac{C}{2-\alpha}\left( \left(\frac{b_n}{a_j}\right)^{2-\alpha}-1\right), & \alpha\in (0,2)\\
	C\log\left(\frac{b_n}{a_j}\right), & \alpha =2\\
	\dfrac{C}{\alpha-2}\left(1 - \left(\frac{a_j}{b_n}\right)^{\alpha-2}\right), & \alpha> 2\\
	\end{cases}\\
	& \leq \begin{cases}
	\dfrac{C}{2-\alpha}\left(\frac{b_n}{a_j}\right)^{2-\alpha}, & \alpha\in (0,2)\\
	C\left(\frac{b_n}{a_j}\right), & \alpha =2\\
	\dfrac{C}{\alpha-2}, & \alpha>2.
	\end{cases}
	\end{align*}
	The desired result follows by combining the latter expression with \eqref{boundsumE}.
\end{proof}

\begin{theorem}
\label{weaknew1}Let $(R_n)_{n\geq 1}$ be as in \eqref{Rdef} and assume that the assumptions of Lemma \ref{lemma1} are satisfied for some $\alpha>0$ and let
\begin{equation*}
n\sum_{j=1}^{n}\left(\frac{a_j}{b_n}\right)^{\alpha\wedge 1}\to 0\quad \mbox{as}\quad n\to \infty.
\end{equation*}
Then,
\[
\frac{1}{b_n}\sum_{j=1}^{n}a_j(R_j-ER_{nj}) \mathop {\longrightarrow }^P 0,\quad n\to \infty.
\]	
\end{theorem}
\begin{proof}
First, we write  
\[
\frac{1}{b_n}\sum_{j=1}^{n}a_j(R_j-ER_{nj})=\frac{1}{b_n}\sum_{j=1}^{n}a_j(R_j-R_{nj})+\frac{1}{b_n}\sum_{j=1}^{n}a_j(R_{nj}-ER_{nj}).
\]
Thus, the result will follow by proving that, as $n\to \infty$,
\begin{equation}
\label{conv1a}\frac{1}{b_n}\sum_{j=1}^{n}a_j(R_j-R_{nj}) \mathop {\longrightarrow }^P 0
\end{equation}
and
\begin{equation}
\label{conva2}\frac{1}{b_n}\sum_{j=1}^{n}a_j(R_{nj}-ER_{nj})\mathop {\longrightarrow }^P 0.
\end{equation}
For \eqref{conv1a} we have that $\forall \epsilon>0$ 
\[
P\left(\frac{1}{b_n}\left | \sum_{j=1}^{n}a_j(R_j-R_{nj}) \right |>\epsilon\right)\leq P\left(\displaystyle\cup_{j=1}^{n}\{R_j \neq R_{nj}\}\right)\leq \sum_{j=1}^{n} P\left(R_j>\frac{b_n}{a_j}\right) \to 0,\quad n\to \infty
\]
because of Lemma \ref{lemma1}.  
By utilizing Markov inequality we have that
\begin{align*}\
&P\left(\frac{1}{b_n}\left | \sum_{j=1}^{n}a_j(R_{nj}- ER_{nj}) \right |>\epsilon\right)\leq\frac{1}{b_n^2\epsilon^2}E\left(   \sum_{j=1}^{n}a_j(R_{nj}- ER_{nj})      \right)^2\\
&\leq
\begin{cases}
\displaystyle
Cn\sum_{j=1}^{n}\frac{a_j^\alpha}{b_n^{\alpha}}, & \alpha\in (0,1]\\
\displaystyle
Cn\sum_{j=1}^{n}\frac{a_j}{b_n}, & \alpha =2\\
\displaystyle Cn\sum_{j=1}^{n}\frac{a_j^2}{b_n^2} & \alpha>2.
\end{cases}
\end{align*}
Recall that for sufficiently large $n$ we have that $ \frac{a_j}{b_n}\leq 1$. Hence, for sufficiently large $n$,  and $\alpha \in (0,1]$  
\[
\left( \frac{a_j}{b_n}\right)^2\leq \left( \frac{a_j}{b_n}\right)\leq \left( \frac{a_j}{b_n}\right)^\alpha,
\]
for $\alpha \in [1,2]$
\[
\left( \frac{a_j}{b_n}\right)^2\leq \left( \frac{a_j}{b_n}\right)^\alpha\leq \left( \frac{a_j}{b_n}\right),
\]
while for $\alpha\geq 2$
\[
\left( \frac{a_j}{b_n}\right)^\alpha\leq \left( \frac{a_j}{b_n}\right)^2\leq \left( \frac{a_j}{b_n}\right).
\]
As a result, the probability calculated above tends to 0 when $n\to\infty$ according to the given set of conditions. 
\end{proof}

\medskip

\noindent The weak law that follows is obtained without any convergence assumption for the sequence of distribution functions $(F_n)_{n\geq 1}$. For the needs of the result, we recall here the definition of a slowly varying function at infinity. For more on slowly varying functions, the interested reader may refer to \cite{BGT1989}.

\begin{definition}
	A positive measurable function $f(x)$ defined on $[c,\infty)$ with $c\geq 0$ is said to be slowly varying at infinity if
	\[
	\lim_{x\to\infty}\frac{f(tx)}{f(x)} =1 \quad \mbox{for all} \quad t>0.
	\]
\end{definition}

\medskip

\begin{theorem}
\label{weaknew2}Let $\{R_n\}_{n\geq 1}$ be as in \eqref{Rdef} and $\{Y_n\}_{n\geq 1}$ as in \eqref{stdom} and assume that the function $H_j(x) = E(Y_jI(Y_j\leq x))$ is a slowly varying function at infinity for any integer $j$. Let $\{c_{nj}, 1\leq j\leq m_n, n\geq 1\}$ be a triangular array of real numbers such that
\begin{enumerate}
	\item[a.] $|c_{nj}|\leq 1$ for any integer $n$ and $1\leq j \leq m_n$;
	\item[b.]$m_n\sum_{j=1}^{m_n}c_{nj}^2\to0$ as $n\to \infty$;
	\item[c.] $\sup_{n} m_n\sum_{k=1}^{m_n}|c_{nj}|H_j(|c_{nj}|^{-1})<\infty$.
\end{enumerate}
Then, as $n\to \infty$,
\[
\sum_{j=1}^{m_n}c_{nj}(R_j - E(R_jI(|c_{nj}R_j|\leq 1))) \mathop {\longrightarrow }^P 0.
\]
\end{theorem}
\begin{proof}
	For $n\geq 1$ and $1\leq j\leq m_{n}$ we define
	\[
	R'_{nj} = R_jI(|c_{nj}R_j|\leq 1),\quad R''_{nj} = R_jI(|c_{nj}R_j|>1)\quad
	\mbox{and}\quad S'_n = \sum_{j=1}^{m_n}c_{nj}(R'_{nj}-ER'_{nj}).
	\]	
First, observe that
\begin{eqnarray*}
	\sum_{j=1}^{m_n}c_{nj}(R_j - E(R_jI(|c_{nj}R_j|\leq 1))) &=& \sum_{j=1}^{m_n}c_{nj}(R_jI(|c_{nj}R_j|\leq 1) - E(R_jI(|c_{nj}R_j|\leq 1)))\\
	&~~~~~+&\sum_{j=1}^{m_n}c_{nj}R_jI(|c_{nj}R_j|> 1)\\
	&=&\sum_{kj=1}^{m_n}c_{nj}(R'_{nj}-ER'_{nj}) +\sum_{j=1}^{m_n}c_{nj}R''_{nj} = S'_n + \sum_{j=1}^{m_n}c_{nj}R''_{nj}.
\end{eqnarray*}
Thus, the desired result will follow by proving that as $n \to \infty$ both summands tend to 0 in probability.
For the first term in the latter expression, note that for any $\epsilon>0$ we can write	
\begin{align*}
&P(|S_n'|>\epsilon) \leq \frac{E|S_n'|^2}{\epsilon^2} = \epsilon^{-2}E\left(\sum_{j=1}^{m_n}c_{nj}(R'_{nj}-ER'_{nj}) \right)^2\leq \epsilon^{-2}m_n\sum_{j=1}^{m_n}c^2_{nj}E(|R'_{nj}|^2)\\
&=m_n\epsilon^{-2}\sum_{j=1}^{m_n}c^2_{nj}E(R_j^2I(|c_{nj}R_j|\leq 1)) = m_n\epsilon^{-2}\sum_{j=1}^{m_n}c^2_{nj}E(R_j^2I(R_j\leq |c_{nj}|^{-1}))\\
&\leq m_n\epsilon^{-2}\sum_{j=1}^{m_n}c^2_{nj}(1+|c_{nj}|^{-2}P(Y_j>|c_{nj}|^{-1})+E(Y_j^2I(Y_j\leq |c_{nj}|^{-1}))),
\end{align*}	
by Lemma \ref{trunmom}.  Because of the assumption that $H_j(x)$ is a slowly varying function at infinity for every $j$,  Lemma 3 in \cite{DSY2018} ensures that for any $\delta>0$ there exists $x_0>0$ such that for $x>x_0$ and any integer $j$,
\[
P(Y_j>x)\leq \delta x^{-1}H_j(x)\quad\mbox{and}\quad E(Y_j^2I(Y_j\leq x))\leq \delta xH_j(x).
\]
Notice that because of the second assumption, as $n\to\infty$,
\[
|c_{nj}| \leq \left(m_n\sum_{j=1}^{m_n}c_{nj}^2\right)^{1/2} \to 0
\]
so, for sufficiently large $n$ and every $1\leq j\leq m_n$ we have 
\[
|c_{nj}|^{-1}>x_0.
\]
\noindent Thus, for any large enough $n$ we have for all $1\leq j\leq m_n$
\[
P(Y_j>|c_{nj}|^{-1})\leq \delta |c_{nj}|H_k(|c_{nj}|^{-1})\quad\mbox{and}\quad E(Y_j^2I(Y_j\leq |c_{nj}|^{-1}))\leq \delta |c_{nj}|^{-1}H_j(|c_{nj}|^{-1}).
\]	
Thus, 
\[
P(|S_n'|>\epsilon) \leq m_n\epsilon^{-2}\left(\sum_{k=1}^{m_n}c^2_{nk}+ 2\sum_{k=1}^{m_n} \delta |c_{nk}|H_k(|c_{nk}|^{-1})\right) \to 0,
\]
if we allow $n\to \infty$ and then $\delta \to 0$. Furthermore, for $\epsilon>0$ and due to the stochastic dominance of $R_n$ from $Y_n$, we have
\begin{align*}
&P\left(\left |\sum_{k=1}^{m_n}c_{nk}R''_k\right |>\epsilon \right)\leq \sum_{k=1}^{m_n}P(|c_{nk}R_k|>1) = \sum_{k=1}^{m_n}P(R_k>|c_{nk}|^{-1}) \\
&\leq \sum_{k=1}^{m_n}P(Y_k>|c_{nk}|^{-1}) \leq \delta \sum_{k=1}^{m_n}  |c_{nk}|H_k(|c_{nk}|^{-1})\to 0,
\end{align*}
as $n\to \infty$ and by letting $\delta \to 0$. 
\end{proof}

\begin{remark}
	It can easily be verified that under the assumptions of Theorems 4 and 5 in \cite{GH2020}, the quantity $EY_kI(Y_k\leq x)$ is a slowly varying function at infinity.
\end{remark}

\section{A strong law of large numbers}

In this section we provide a strong law of large numbers for the sequence $(R_n)_{n\geq 1}$.

\begin{theorem} Let $\{R_n\}_{n\geq 1}$ be as in \eqref{Rdef} with $(F_n)_{n\geq 1}$ satisfying conditions \eqref{F1cond} and \eqref{F2cond} for $\alpha = 1$. Then, for $\beta>0$ and $p\geq 2$,
	\begin{equation}
	\label{stlaw}\frac{1}{\rho(n)\log^\beta n}\sum_{j=1}^{n}\frac{\log^{\beta-p}j}{j}R_j
 \to 0	\quad \mbox{a.s.}
 \end{equation}
 for every sequence  $\rho(n)$ such that  $\sum_{n=1}^{\infty}1/\rho^2(n)$ converges.
\end{theorem}
\begin{proof}
	Let $a_j = \displaystyle \frac{\log^{\beta-p}j}{j}$, $b_j = \log^{\beta}j$ and $\displaystyle c_j = \frac{b_j}{a_j} = j \log^p j$. We define the sequence of functions  $g_n: [1,\infty)\to [0,\infty)$ 
	\[
	g_n(x) = xI(x\leq c_n)+c_nI(x>c_n)
	\]
	and observe that the LHS of \eqref{stlaw} can be equivalently written as 
	\begin{eqnarray*}
	\frac{1}{\rho(n)b_n}\sum_{j=1}^{n}a_jR_j &=& \frac{1}{\rho(n)b_n}  \sum_{j=1}^{n}a_j(g_j(R_j)-E(g_j(R_j)))\\
	&~~~~~~~&+\frac{1}{\rho(n)b_n}\sum_{j=1}^{n}a_j[ R_jI(R_j>c_j)-c_jI(R_j>c_j)] \\
	&~~~~~~~&+\frac{1}{\rho(n)b_n}\sum_{j=1}^{n}a_j [ER_jI(R_j\leq c_j)+c_jP(R_j>c_j)]\\
	&=:&I_1+I_2+I_3.
	\end{eqnarray*}
We start by studying $I_2$. 
\begin{eqnarray*}
|I_2| &=& \left|\frac{1}{\rho(n)b_n}\sum_{j=1}^{n}a_j[ R_jI(R_j>c_j)-c_jI(R_j>c_j)]\right|\\
&\leq& \frac{1}{\rho(n)b_n}\sum_{j=1}^{n}a_j(R_jI(R_j>c_j)+R_jI(R_j>c_j))\\
&=& \frac{2}{\rho(n)b_n} \sum_{j=1}^{n}a_jR_jI(R_j>c_j).
\end{eqnarray*}

\medskip

\noindent Notice that $\exists n_0$ such that for $\forall j>n_0, \, c_j>1$ so \eqref{stdom} holds true with $x = c_j$. Thus, 
\begin{eqnarray*}
\sum_{j=1}^{\infty}P(R_j>c_j) &=& \sum_{j=1}^{n_0}P(R_j>c_j) + \sum_{j=n_0+1}^{\infty}P(R_j>c_j)\\
& \leq& \sum_{j=1}^{n_0}P(R_j>c_j)+\sum_{j=n_0+1}^{\infty}P(Y_j>c_j).
\end{eqnarray*}
The first term is finite while for the second term, because of \eqref{F2cond} we have that
\[
xP(Y_n>x) \to c \quad x\to \infty,
\]
so there is constant $C$ such that for sufficiently large $x$, $P(Y_n>x)<C/x$. Thus, there exists $n_1$ such that $P(Y_j>c_j) \leq \frac{c}{jlog^pj}$ for $j>n_1$. Hence, by taking $n_0>n_1$ we get
\[
\sum_{j=1}^{\infty}P(R_j>c_j) \leq\sum_{j=1}^{n_0}P(R_j>c_j)+\sum_{j=n_0+1}^{\infty}\frac{C}{j\log^p j}<\infty.
\]
Hence, by Borel-Cantelli we get $P(R_j< c_j \, \, \mbox{ultimately})=1$ and therefore
\[
|I_2|\leq \frac{2}{\rho(n)b_n}\sum_{j=1}^{n}a_j R_jI(R_j>c_j) \to 0\quad\mbox{a.s.}
\]
since $\rho(n)b_n\to \infty$, which implies $\frac{2}{\rho(n)b_n}\sum_{j=1}^{n}a_j R_jI(R_j>c_j) \to 0$.
\noindent For $I_3$ we start by defining
\[
J = \frac{1}{\rho(n)b_n}\sum_{j=1}^{n} a_jc_jP(R_j>c_j).
\]
For some positive integer $n_0$ and for any $j> n_0$, we have $c_j>1$. Thus, 
\begin{eqnarray*}
J&=&\sum_{j=1}^{n_0}\frac{b_j}{\rho(n)b_n}P(R_j>c_j)+\sum_{j=n_0+1}^n\frac{b_j}{\rho(n)b_n}P(R_j>c_j)\\
&\leq& \sum_{j=1}^{n_0}\frac{b_j}{\rho(n)b_n} P(R_j>c_j) +\sum_{j=n_0+1}^n\frac{b_j}{\rho(n)b_n}P(Y_j>c_j)\\
&=:&J_1+J_2
\end{eqnarray*}
where the second inequality is due to \eqref{stdom}. Clearly, as $n\to \infty$ the first term $J_1\to 0$. For the second term we use again the relation $P(Y_n>x)<C/x$ and  we have
\[
J_2 \leq \frac{C}{\rho(n)b_n}\sum_{j=n_0+1}^n\frac{b_j}{c_j}=\frac{C}{\rho(n)b_n}\sum_{j=n_0+1}^n \frac{\log^{\beta-p}j}{j}.
\]

\noindent
By Cesaro's theorem the sequence
\begin{align*}&
\frac{1}{ b_n}\sum_{j=n_0+1}^n \frac{\log^{\beta-p}j}{j}= \frac{1}{\log^{\beta }n  }\sum_{j=n_0+1}^n \frac{\log^{\beta-p}j}{j}
\end{align*}
has the same limit as
\begin{align*}&
\frac{\log^{\beta-p}n}{n}\cdot \frac{1}{\log^{\beta }n-\log^{\beta }(n-1)}\sim \frac{ \log^{\beta-p}n}{\beta \log^{\beta -1}n}= \frac{1}{\beta\log^{p-1}n}\to 0, \qquad n\to\infty,
\end{align*}
which leads to the conclusion that $J\to 0$ as $n\to\infty$. 

\medskip

\noindent Next, we study the term $ER_jI(R_j\leq c_j)$. From \eqref{trun1} in Lemma \ref{trunmom},  again by  the relation $P(Y_n>x)<C/x$, we have that for $c_j> 1$ 
\begin{eqnarray*}
ER_jI(R_j\leq c_j) &\leq&1+ c_jP(Y_j>c_j)+E(Y_jI(Y_j\leq c_j))\\
&\leq&C + \int_{1}^{c_j}P(Y_j>x){\rm d}x\\
&\leq& C + C_1\int_{1}^{c_j} \frac{1}{x}{\rm d}x\\
&=& C(1+\log c_j)\\
&\sim& C\log j,
\end{eqnarray*}
where the last result can be easily obtained by observing that $\log c_j = \log j + p \log \log j \sim \log j$. Consider $n_0$ such that for all $j>n_0,\, c_j >1$. Then, as $n\to \infty$,
\[
\sum_{j=1}^{n}\frac{a_j}{\rho(n)b_n}E(R_jI(R_j\leq c_j)) \sim \frac{1}{\rho(n)(\log n)^{\beta}}\sum_{j=1}^{n_0}a_jER_jI(R_j\leq c_j)+\frac{C}{\rho(n)(\log n)^{\beta}}\sum_{j=n_0+1}^{n}\frac{(\log j)^{\beta-p+1}}{j} \to  0
\]
where the convergence to 0 follows since
$$\frac{1}{\rho(n)(\log n)^{\beta}}\sum_{j=n_0+1}^{n}\frac{(\log j)^{\beta-p+1}}{j}\sim \frac{(\log n)^{\beta-p+2}}{\rho(n)(\log n)^{\beta}}= \frac{1}{ \rho(n)(\log n)^{p-2}}\to 0,\quad n\to \infty. $$
\medskip

\noindent For $I_1$ we need to prove that 
\[
\frac{1}{\rho(n)b_n}\sum_{j=1}^{n}a_j(g_j(R_j)-E(g_j(R_j))) \to 0\quad\mbox{a.s.}
\]
The arguments are similar to the ones used in \cite{GH2020}. Define $S_n = \displaystyle\sum_{j=1}^{n}a_jg_j(R_j)$. It is sufficient to prove that $\forall \epsilon>0$
\begin{equation}
\label{complconv}\sum_{n=1}^{\infty}P\left(\frac{1}{\rho(n)b_n}|S_n-ES_n|>\epsilon\right)<\infty
\end{equation}
and the desired result will follow by applying the Borel-Cantelli lemma.  In order to prove the convergence result described in \eqref{complconv}, the Chebyshev's inequality will be employed. Thus,  it is sufficient to prove that
\begin{equation}
\label{var1}\sum_{n=1}^{\infty}\frac{\Var S_n}{\rho^2(n) b_n^2} <\infty.
\end{equation}
 Conditions \eqref{F1cond} and \eqref{F2cond} (for $\alpha =1$) ensure that 
\[
|\Cov(g_i(R_i),g_j(R_j))|\leq C(\log i+ \log j) 
\]
and
\[
\Var g_j(R_j) \leq Cc_j= Cj(\log j)^p.
\]
where $C$ is a positive constant (Lemmas 4 and 5 respectively in \cite{GH2020}). Using the above expressions for $W_j = \displaystyle \frac{(\log j)^{\beta-p}}{j}g_j(R_j)$  we have that
\[
|\Cov(W_i,W_j)|\leq \frac{C(\log i\log j)^{\beta-p}}{ij}(\log i + \log j)
\]
and 
\begin{eqnarray*}
\sum_{j=1}^{n}\Var(W_j) &=& \sum_{j=1}^{n}\Var\left( \frac{(\log j)^{\beta-p}}{j}g_j(R_j)\right)\\
&=&\sum_{j=1}^{n} \frac{(\log j)^{2\beta-2p}}{j^2}\Var(g_j(R_j))\\
&\leq&\sum_{j=1}^{n} \frac{(\log j)^{2\beta-p}}{j}\\
&<&C(\log n)^{2\beta-p+1}.
\end{eqnarray*}
Moreover,
\begin{eqnarray*}
\sum_{1\leq i<j\leq n}|\Cov (W_i,W_j)|&\leq& C\sum_{j=2}^{n}\sum_{i=1}^{j-1}\frac{(\log i\log j)^{\beta-p}}{ij}(\log i + \log j)\\
&<&C(\log n)^{2\beta -2p+3}.
\end{eqnarray*}
Thus,
\[
\Var S_n = \sum_{j=1}^{n}\Var(W_j) + 2\sum_{1\leq i<j\leq n}\Cov (W_i,W_j) < C(\log n)^{2\beta-p+1} + C(\log n)^{2\beta -2p+3}< C(\log n)^{2\beta-p+1}.
\]
Hence,
\[
\sum_{n=1}^{\infty}\frac{\Var S_n}{\rho^2(n) b_n^2}< C\sum_{n=1}^{\infty} \frac{(\log n)^{2\beta-p+1} }{\rho^2(n)(\log n)^{2\beta}} = C\sum_{n=1}^{\infty}\frac{1}{\rho^2(n)( \log n)^q}<\infty,\, q= p-1\geq 1,
\]
 which leads to the desired result.
\end{proof}

\begin{remark}
It is worth mentioning that a result on almost sure convergence for the weighted partial sums of $\{R_n\}_{n\geq 1}$ can also be found in Theorem 11 of \cite{GH2020} but for different sequences of positive numbers.
\end{remark}

\noindent\textbf{Acknowledgement}
	Prof. R. Giuliano wishes to thank UCLan Cyprus and the University of Cyprus for their hospitality; part of the present paper was completed during her visit in Cyprus.

\end{document}